\documentclass[12pt]{article}
\usepackage{amsmath}
\usepackage{amsthm}
\usepackage{amssymb}
\usepackage{bm}
\usepackage{comment}
\usepackage{enumitem}  
\usepackage{tikz}
\usetikzlibrary{matrix,arrows,decorations.pathmorphing}
\usepackage{tikz-cd}
\newtheorem{mythe}{Theorem}

\newtheorem{mylem}{Lemma}

\newtheorem{mypro}{Proposition}

\title{A choiceless box game paradox}
\author{Elliot Glazer}
\begin{document}

%\centerline{A choiceless box game paradox}
\maketitle

\noindent \textbf{Abstract}

We identify a choiceless variation of the box game paradox, in which players predict unknown real numbers with near-perfect accuracy despite lacking
any useful information. We also verify that choice is necessary in the solution of the original problem and consider some further variations of the game.

\bigskip

\noindent \textbf{1. Introduction}

The axiom of choice has many implications which conflict with intuitions about randomness, particularly the belief that one cannot predict a number they
know nothing about. Two striking examples have spread over multiple internet forums: the infinite hat problem of Gabay and O'Connor ([1], [4]), and its spiritual
successor the box game problem,\footnote{I have not been able to identify who originally came up with box game problems. The earliest reference I could find is
[3]. See also [6] and [7].} which we recount here.
%check ref

A team of 100 players wait outside a room in which there is a countably infinite sequence of boxes, each containing a real number.
One at a time, each player will enter the room and open as many boxes as they like, even
%check not plagiar
infinitely many, but must eventually guess the value within some unopened box. They then close the boxes and leave. There is no communication between the players once the game begins. The team wins if at least 99 of them make a correct guess. Paradoxically, there is a strategy which ensures the team wins,
suggesting each player has a ``99\% chance" of guessing a real number of which they have no information.

Here is one solution to this puzzle. For sequences $s, t \in \mathbb{R}^{\mathbb{N}},$ denote $s \sim_{\text{tail}} t$ if they have the same tail, i.e.,
there is $n$ such that for all $m \ge n,$ $s(m)=t(m).$ Then $\sim_{\text{tail}}$ is an equivalence relation. Choose one sequence from each equivalence class.

Let $v \in \mathbb{R}^{\mathbb{N}}$ be the sequence of reals in the boxes. For $p \in [100],$ define $s_p \in \mathbb{R}^{\mathbb{N}}$ by
$s_p(i)=v(100i+p)$ and let $\hat{s}_p$ be the $\sim_{\text{tail}}$ class representative for $s_p.$ Let $i_p \in \mathbb{N}$ be least such that
$s_p(i) = \hat{s}_p(i)$ for all $i \ge i_p$ and let $k_p= \max(\{i_n: n \in [100], n \neq p\}).$

We now describe the strategy of player $p.$ First, open all boxes except those congruent to $p$ mod 100 in order to
determine $s_n,$ $\hat{s}_n,$ and $i_n$ for all $n \neq p,$ as well as $k_p.$ Now open all remaining boxes except for
box $100k_p+p$ in order to determine the sequence $\hat{s}_p.$ Finally, guess that the value in box $100k_p+p$ is $\hat{s}_p(k_p).$ This
strategy can only fail for a $p$ which strictly maximizes $i_p,$ so at least 99 players succeed.

Due to the use of choice in picking representatives from the equivalence classes, we have not provided an explicit winning strategy
for this box game. In fact, we will show in Section 3 that choice is necessary to prove the existence of a winning strategy. However, in Section 4 we will
identify a variation of the box game which has an explicit winning strategy, simply by increasing the number of boxes in the room.

\bigskip

\noindent \textbf{2. Notation}

We will consider several versions of the box game, which vary in the sets of players, boxes, and the possible values in the boxes, but all having the win
condition that at most one player fails to make a correct guess (either by guessing wrong or by opening all boxes without making a guess).

Let $G(P, B, V)$ be the variant of the box game with set of players $P,$ set of boxes $B,$ and $V$ the set of possible values in the boxes (e.g.,
the original box game is $G([100], \mathbb{N}, \mathbb{R})$). For each player, a move consists either of (i) opening up a subset of the boxes, or of (ii)
choosing an ordered pair $(b, v),$ which denotes guessing box $b$ has value $v$ in it. A player may make a transfinite sequence of moves of type
(i) before making their guess, though the strategies we construct will involve only a bounded finite sequence of moves.

We denote the team's collective strategy by $\sigma$ and the strategy of player $p$ by $\sigma_p.$ The inputs for $\sigma_p$ are the history of moves
the player has made and the values in the boxes they have seen, and the output is the player's next move.

We say $G(P, B, V)$ is \textit{winnable} if there a strategy $\sigma$ which ensures at most
one player makes an incorrect guess. Notice that only the cardinalities of $B,$ $P,$ and $V$ are relevant to whether $G(P, B, V)$ is winnable.

\bigskip

\noindent \textbf{3. Choice is necessary}

We will first show that some amount of choice is necessary to prove existence of winning strategies for the original box game by verifying this for an
easier game:

\begin{mythe} It is consistent relative to ZF set theory that $G([3], \mathbb{N}, \{0,1\})$ is not winnable.\end{mythe}

We assume there is a total translation-invariant probability measure $\mu$ on the unit interval, a hypothesis which is consistent relative to ZF.\footnote{Relative
to an inaccessible cardinal, it is consistent that all sets of reals are Lebesgue measurable. We adopt a weaker hypothesis that avoids this large cardinal
assumption, cf. [5, 15.1].}

Fix a player $p$ and their strategy $\sigma_p.$ Choose a real $r$ in the interval at random and use its binary expansion as the sequence of bits in the boxes.
Let $E_1$ be the event that $p$ guesses correctly and $E_2$ that $p$ guesses wrong.\footnote{Note that $E_1$ and $E_2$ need not be
exhaustive since $p$ can fail by opening all boxes without guessing.}
We will show $\mu(E_1)=\mu(E_2).$

Let $n(p)$ be the box $p$ guesses from and $g(p)$ be the value they guess to be in that box.
If $p$ opens all boxes without making a guess, we set $n(p)=g(p)=-1.$ Let $v(n)$ be the value in box $n,$ i.e. the $n$th bit of $r.$ Then

\begin{align*}
\mu(E_1) &= \sum_{n \in \mathbb{N}} \mu(n=n(p), g(p)=0, v(n)=0)
+\sum_{n \in \mathbb{N}} \mu(n=n(p), g(p)=1, v(n)=1)
\\&= \sum_{n \in \mathbb{N}} \mu(n=n(p), g(p)=0, v(n)=1)
+\sum_{n \in \mathbb{N}} \mu(n=n(p), g(p)=1, v(n)=0)
\\&=\mu(E_2).
\end{align*}

Therefore, $\mu(E_1) \le \frac{1}{2}.$ Thus, the expected number of players to fail is at least $\frac{3}{2},$
so there exists a box sequence for which at least 2 players fail.

\bigskip

\noindent \textbf{4. A choiceless box game}

Our main result concerns the box game $G(\mathbb{N}, \mathcal{P}(\mathbb{R}), \mathbb{R}).$
This version of the game appears harder than the original, since now infinitely many players have to guess a real number without acquiring any
information about it, and still only one player is allowed to fail. That there are more boxes in the room doesn't seem helpful, yet we can now construct an explicit winning strategy:

\begin{mythe} (ZF) The box game $G(\mathbb{N}, \mathcal{P}(\mathbb{R}), \mathbb{R})$ is winnable.
\footnote{A reader interested in reverse mathematics may check that the theorem can be formalized and proven in $Z_3P,$
third-order arithmetic with an additional unary predicate in the language and the comprehension scheme extended to include formulae involving this
predicate. The predicate encodes the reals inside the boxes. Cf. [2] for how to formalize the transfinite recursion in Lemma 1 in
this framework.}
\end{mythe}
%ref
We begin by identifying the set of boxes with $\mathcal{P}(\mathbb{R}) \times \mathbb{N} \times \mathbb{N},$ and the values hidden in the boxes by a
function $v: \mathcal{P}(\mathbb{R}) \times \mathbb{N} \times \mathbb{N} \rightarrow \mathbb{R}.$
For $h_i \in \mathbb{R}^{\mathbb{N} \times \mathbb{N}},$ define $h_1 \sim_{\text{col}}h_2$ if they agree on all but finitely many columns, i.e., there is
$n$ such that for all $i \ge n$ and $j \in \mathbb{N},$ $h_1(i, j)=h_2(i, j).$ Notice $\sim_{\text{col}}$ and $\sim_{\text{tail}}$ are each an equivalence
relation on a set in bijection with $\mathbb{R}.$

We now isolate a lemma which involves our only use of transfinite recursion:

\begin{mylem}
Let $\sim$ be an equivalence relation on $\mathbb{R}.$ There is an explicit map which sends every
$f: \mathcal{P}(\mathbb{R}) \rightarrow \mathbb{R}/\sim$ to an $(X, Y)$ such that $X \neq Y$ and $f(X)=f(Y).$
\end{mylem}

\begin{proof}
By transfinite recursion, define $X_{\alpha}=\bigcup_{\beta<\alpha} f(X_{\beta}).$ This sequence stabilizes precisely
at the least $\alpha$ such that $f(X_{\alpha})=f(X_{\beta})$ for some $\beta<\alpha,$ so $X=X_{\alpha}$ and $Y=X_{\beta}$ are as desired.
\end{proof}

Define $f: \mathcal{P}(\mathbb{R}) \rightarrow (\mathbb{R}^{\mathbb{N} \times \mathbb{N}})/\sim_{\text{col}}$ by
$f(X) = [(n, i) \mapsto v(X, n, i)]_{\sim_{\text{col}}}$ and $f_n: \mathcal{P}(\mathbb{R}) \rightarrow (\mathbb{R}^{\mathbb{N}})/\sim_{\text{tail}}$
by $f_n(X) = [i \mapsto v(X, n, i)]_{\sim_{\text{tail}}}.$ Let $(X^*, Y^*)$ be the canonical pair (as in Lemma 1) such that $f(X^*)=f(Y^*)$ and
$(X_n, Y_n)$ the canonical pair such that $f_n(X_n) = f_n(Y_n).$ For each $n,$ if there is $i$ such that $v(X^*, n, i') = v(Y^*, n, i')$ for all $i' \ge i,$ let $i_n$ be least such,
and otherwise set $i_n=-1.$ Let $J=\{n: i_n = -1\}$ and let $j_n$ be least such that for all $j \ge j_n,$ $v(X_n, n, j) = v(Y_n, n, j).$ 
For cofinitely many $n,$ $i_n = 1,$ so for each $p \in \mathbb{N}$ we can define
$$k_p = \max(\{i_n: n \neq p\} \cup \{j_n: n \in J \setminus \{p\}\}).$$
%Natural number start

We now describe the strategy $\sigma_p$ for player $p.$
First, open all boxes in $\mathcal{P}(\mathbb{R}) \times (\mathbb{N} \setminus \{p\}) \times \mathbb{N}.$ For $n \neq p,$ this determines $f,$
$f_n,$ $(X^*, Y^*),$ $(X_n, Y_n),$ $J \setminus \{p\},$ and $k_p.$
Open all boxes in $\mathcal{P}(\mathbb{R}) \times \{p\} \times (\mathbb{N} \setminus \{k_p\}).$ This determines whether $p \in J.$

If $p \not \in J,$ then open the box $(X^*, p, k_p),$ and guess that the box $(Y^*, p, k_p)$ contains $v(X^*, p, k_p).$
If $p \in J,$ open the box $(X_p, p, k_p)$ and guess that the box $(Y_p, p, k_p)$ contains $v(X_p, p, k_p).$
This completes the description of strategy $\sigma_p.$

Suppose $\sigma_p$ results in an incorrect guess. If $p \not \in J,$ then $i_p>k_p.$ If $p \in J,$ then $j_p > k_p.$ In either case, for all $p' \neq p,$
we have $k_{p'}>k_p.$ Thus, at most one player will fail, so the strategy $\sigma$ is as desired.

\bigskip

\noindent \textbf{5. More box games}

We will briefly consider which further variations of the box game have winning strategies in the absence of choice.
We first note that for any set of players $P$ and set of values $V,$ there is a sufficiently large set of boxes such that
$G(P, B, V)$ is winnable. In particular, the proof of Theorem 2 immediately generalizes to the following:

\begin{mypro}
(ZF) For any $P$ and $V,$ the box game $G(P, \mathcal{P}(V^{P \times \mathbb{N}}), V)$ is winnable.
\end{mypro}

One remaining question is how many boxes are enough for a choiceless box game paradox, with say three players and the boxes containing bits.
The best known unconditional result is the following:

\begin{mypro}
(ZF) The box game $G([3], \{X \subset \mathbb{R}: |X|<|\mathbb{R}|\}, \{0,1\})$ is winnable.
\end{mypro}
\noindent This can be verified by casework on whether $\mathbb{R}$ is well-orderable: if it is, countably many boxes suffice by the original box game paradox. Otherwise, $\mathcal{P}_{\text{WO}}(\mathbb{R}):=\{X \subset \mathbb{R}: X \text{ is well-orderable}\} \subset B.$ Modify the proof of Lemma 1
to construct a map which sends every $f: \mathcal{P}_{\text{WO}}(\mathbb{R}) \rightarrow \{0,1\}^{\mathbb{N}} / \sim_{\text{tail}}$ to an $(X, Y)$
such that $X \neq Y$ and $f(X)=f(Y),$ and then proceed as in the proof of Theorem 2.

There are many open questions regarding games with a continuum of boxes, e.g.

\bigskip

\noindent \textbf{Question.} \textit{Does ZF prove that } $G([3], \mathbb{R}, \{0,1\})$ \textit{is winnable?}

\bigskip

The author's dissertation will analyze many more box games and their relationship with various set theoretic axioms.

\bigskip

\noindent \textbf{Acknowledgments}

The author is grateful to Dan Velleman for advice on the presentation of the choiceless box game strategy, and to Stan Wagon and Alan Taylor
for insightful discussions on paradoxical puzzles.

\bigbreak

\noindent \textbf{References}

\bigskip

\noindent [1]  \thinspace C. Hardin, A. Taylor, An Introduction to Infinite Hat Problems, \textit{The} \newline \indent \textit{Math. Intelligencer,} \textbf{30} (2008) 20-25.

\noindent [2] \thinspace A. Kanamori, The Mathematical Import of Zermelo's Well-Ordering \newline \indent Theorem, \textit{Bull. Symb. Logic,} \textbf{3} (1997) 281-311.

\noindent [3] \thinspace D. Madore, Le Docteur No continue ses méfaits, \textit{David Madore's WebLog,} \newline
\indent 2009. http://www.madore.org/~david/weblog/d.2009-05-16.1643.html

\noindent [4] \thinspace G. Muller, The Axiom of Choice is Wrong, \textit{The Everything Seminar,} 2007.
\indent https://cornellmath.wordpress.com/2007/09/13/the-axiom-of-choice-is-wrong/

\noindent [5] \thinspace G. Tomkowicz, S. Wagon, \textit{The Banach-Tarski Paradox,} 2nd ed., New \newline \indent York: Cambridge, 2016.

\noindent [6] \thinspace Predicting Real Numbers, Math. Stack Exchange, 2013. \newline
\indent https://math.stackexchange.com/questions/371184/predicting-real-numbers

\noindent [7] \thinspace Probabilities in a riddle involving axiom of choice, MathOverflow, 2013. \newline
\indent https://mathoverflow.net/questions/151286/probabilities-in-a-riddle-involving-\indent axiom-of-choice

\bigskip

\noindent \textit{Email address:} eglazer@math.harvard.edu

\end{document}